\documentclass[10in]{amsart}

\usepackage{enumerate,verbatim}
\usepackage[all,2cell,ps]{xy}
\usepackage{mathptmx}
\usepackage[notcite, notref]{}
\usepackage[pagebackref]{hyperref}
\usepackage{nicefrac}
\usepackage{array}
\usepackage{xcolor}
\usepackage{amsthm}
\usepackage[leqno]{amsmath}
\usepackage[pagebackref]{hyperref}
\usepackage{mathptmx}

\usepackage{amsfonts,amsmath,amssymb,amsthm,amscd,amsxtra}
\usepackage{enumerate,verbatim}

\usepackage[all,2cell,ps]{xy}
\usepackage[pagebackref]{hyperref}
\usepackage{todonotes}
\theoremstyle{plain} 

\newtheorem{thm}{Theorem}[section]

\theoremstyle{definition}

\newtheorem{chunk}[thm]{}

\newtheorem{rmk}[thm]{Remark}

\numberwithin{equation}{section}

\newcommand{\fm}{\mathfrak{m}}

\newcommand{\ZZ}{\mathbb{Z}}
\newcommand{\NN}{\mathbb{N}}

\def\CI{\operatorname{\mathsf{CI-dim}}}
\def\G-dim{\operatorname{\mathsf{G-dim}}}

\def\pd{\operatorname{\mathsf{pd}}}

 \DeclareMathOperator{\cx}{cx}
\DeclareMathOperator{\Tr}{Tr}

\def\Ext{\operatorname{Ext}}
\def\Tor{\operatorname{Tor}}
\def\depth{\operatorname{depth}}
\def\pd{\operatorname{pd}}

\def\Hom{\operatorname{Hom}}

\DeclareMathOperator{\md}{\operatorname{\mathsf{mod}}}

\DeclareMathOperator{\edim}{embdim}
\def\pd{\operatorname{\mathsf{pd}}}

\DeclareMathOperator{\rank}{rank}

\def\pd{\operatorname{\mathsf{pd}}}

\def\Tr{\mathsf{Tr}}

\def\urltilda{\kern -.15em\lower .7ex\hbox{\~{}}\kern .04em}
\def\urldot{\kern -.10em.\kern -.10em}\def\urlhttp{http\kern -.10em\lower -.1ex
\hbox{:}\kern -.12em\lower 0ex\hbox{/}\kern -.18em\lower 0ex\hbox{/}}

\newcommand{\numberseries}{\bfseries}   

\newlength{\thmtopspace}                
\newlength{\thmbotspace}                
\newlength{\thmheadspace}               
\newlength{\thmindent}                  

\newtheoremstyle{fixed bf head,slanted body}
                {\thmtopspace}{\thmbotspace}{\slshape}
                {\thmindent}{\bfseries}{.}{\thmheadspace}
                {{\numberseries \thmnumber{#2\;}}\thmname{#1}\thmnote{ (#3)}}

\newtheoremstyle{variable bf head,slanted body}
                {\thmtopspace}{\thmbotspace}{\slshape}
                {\thmindent}{\bfseries}{.}{\thmheadspace}
                {{\numberseries \thmnumber{#2\;}}\thmname{#1}\thmnote{ #3}}

\newtheoremstyle{fixed bf head,upright body}
                {\thmtopspace}{\thmbotspace}{\upshape}
                {\thmindent}{\bfseries}{.}{\thmheadspace}
                {{\numberseries \thmnumber{#2\;}}\thmname{#1}\thmnote{ (#3)}}

\newtheoremstyle{numbered paragraph}
                {\thmtopspace}{\thmbotspace}{\upshape}
                {\thmindent}{\upshape}{}{\thmheadspace}
                {{\numberseries \thmnumber{#2.}}}

\begin{document}

\title{On modules with reducible complexity}

\author[Celikbas, Sadeghi, Taniguchi]
{Olgur Celikbas, Arash Sadeghi, and Naoki Taniguchi}

\address{Olgur Celikbas \\
Department of Mathematics \\
West Virginia University\\
Morgantown, WV 26506-6310 USA}
\email{olgur.celikbas@math.wvu.edu}


\address{Arash Sadeghi\\
School of Mathematics\\
Institute for Research in Fundamental Sciences, (IPM) \\
P.O. Box: 19395-5746, Tehran, Iran}
\email{sadeghiarash61@gmail.com}

\address{Naoki Taniguchi\\
Global Education Center\\
Waseda University\\
1-6-1 Nishi-Waseda, Shinjuku-ku, Tokyo 169-8050, Japan}
\email{naoki.taniguchi@aoni.waseda.jp}
\urladdr{http://www.aoni.waseda.jp/naoki.taniguchi/}

\thanks{Sadeghi's research was supported by a grant from IPM. Taniguchi's research was supported by JSPS Grant-in-Aid for Young Scientists (B) 17K14176 and Waseda University Grant for Special Research Projects 2018K-444, 2018S-202.}



\subjclass[2010]{}

\begin{abstract} In this paper we generalize a result, concerning a depth equality over local rings, proved independently by Araya and Yoshino, and Iyengar. Our result exploits complexity, a concept which was initially defined by Alperin for finitely generated modules over group algebras, introduced and studied in local algebra by Avramov, and subsequently further developed by Bergh.

\end{abstract}

\maketitle

\thispagestyle{empty}

\section{Introduction}

Throughout $R$ denotes a commutative Noetherian local ring with unique maximal ideal $\fm$ and residue field $k$, and $\md R$ denotes the category of all finitely generated $R$-modules.

In this paper we are mainly concerned with the following theorem of Auslander \cite{Au}:

\begin{thm}\label{thmold} (\cite[3.1]{Au}) Let $M, N\in \md R$ be modules, either of which has finite projective dimension. If $\Tor_{i}^R(M,N)=0$ for all $i\geq 1$, then it follows  that $\depth(M)+\depth(N)=\depth(R)+\depth(M\otimes_R N)$.
\end{thm}

Huneke and Wiegand extended Auslander's result, and proved in \cite{HW1} that Tor-independent modules (not necessarily of finite projective dimension) over complete intersection rings also satisfy the depth equality of Theorem \ref{thmold}; such depth equality was dubbed ``the depth formula'' by Huneke and Wiegand in \cite {HW1}. 

The aforementioned result of Huneke and Wiegand was extended -- independently by Araya and Yoshino  \cite{ArY}, and Iyengar \cite{I} -- to the case where the ring in question is local and either of the modules considered has finite complete intersection dimension; see also Christensen and Jorgensen \cite{CJ}, Foxby \cite{Foxby} and Iyengar \cite{I} for extensions of the depth formula to certain complexes of modules. 

The main purpose of this article is to prove an extension of Theorem \ref{thmold}. Our main result is:

\begin{thm}\label{thmintro} Let $M, N\in \md R$ be modules. Assume $\Ext^i_R(M,R)=0$ for $i\gg0$ and $M$ has reducible complexity.  If $\Tor_{i}^R(M,N)=0$ for all $i\geq 1$, then $\depth(M)+\depth(N)=\depth(R)+\depth(M\otimes_R N)$, i.e., the depth formula for $M$ and $N$ holds.
\end{thm}

In the next section, we recall the definition of complexity and that of reducible complexity, and prove Theorem \ref{thmintro} in section 3. Here let us note that the extension of Theorem \ref{thmold} we establish in Theorem \ref{thmintro} seems to be quite different in nature than those exist in the literature: all of the improvements of Theorem \ref{thmold}, which we are aware of, assume the finiteness of a version of a homological dimension of the module in question. On the contrary, in Theorem \ref{thmintro}, what we assume for the module $M$ is not a homological dimension. Moreover, our hypothesis on $M$ is weaker than the condition ``$M$ has finite complete intersection dimension ". In general, if $M$ has finite complete intersection dimension (e.g., $R$ is a complete intersection), then $\Ext^i_R(M,R)=0$ for $i\gg0$ and $M$ has reducible complexity, but not vice versa: there do exist examples of modules $M$ over Gorenstein rings (so that $\Ext^i_R(M,R)=0$ for $i\gg0$) such that $M$ has reducible complexity, but $M$ does not have finite complete intersection dimension; see, for example,  \cite[Example on page 136]{Be}.

\section{Preliminaries}

We refer the reader to \cite{AuBr, AGP} for the definitions of standard homological dimensions, such as the complete intersection dimension, and proceed by recalling the definitions of Auslander transpose and complexity.

\begin{chunk} \label{a1} \textbf{Auslander Transpose.} (\cite[2.8]{AuBr})
Let $M$ be an $R$-module. Then the \emph{transpose} $\Tr M$ of $M$
is given by the exact sequence $0\rightarrow M^*\rightarrow P_0^*\xrightarrow{f^{\ast}} P_1^*\rightarrow \Tr M\rightarrow 0$, where  $(-)^{\ast}=\Hom_R(-,R)$ and $P_1 \xrightarrow {f}  P_0\rightarrow M\rightarrow 0$ is a projective presentation of $M$. Notice, $\Tr M$ is unique, up to projectives. Moreover, there is an exact sequence of functors of the form:
\vspace{-0.1in}
\begin{align} \tag{\ref{a1}.1}
0\rightarrow\Ext^1_R(\Tr\Omega^nM,-)\rightarrow\Tor_n^R(M,-)\rightarrow\Hom_R(\Ext^n_R(M,R),-)
\rightarrow\Ext^2_R(\Tr\Omega^nM,-).   \;\;\;\;\;\; \;\;\;\;\; \;\;\; \qed 
\end{align} 



\end{chunk}





\begin{chunk} \textbf{Complexity.} \label{cx} (\cite{Alp, AE, Av1, AvBu}) If $B =\{b_i\}_{i\geq 0}$ is a sequence of nonnegative integers, then the \textit{complexity} of the sequence $B$ is  $\cx(B) = \inf\{r\in \NN \cup \{0\} \mid
b_n\leq A \cdot n^{r-1} \ \text {for some real number} \  A  \ \text
{and for all} \ n\gg 0 \}$. 

The complexity $\cx(M,N)$ of a pair of modules $M,N \in \md R$ is $\cx\left( \{ \rank_{k}(\Ext^{i}_{R}(M,N)\otimes_{R}k) \} \right)$. Then the complexity $\cx(M)$ of $M$ equals $\cx(M,k)$ so that it is a measure on a polynomial scale of the growth of the ranks of the free modules in its minimal free resolution; see \cite{AvBu}. If $M\in \md R$ has finite complete intersection dimension (e.g., $R$ is a complete intersection), then $\cx_R(M) \leq \edim(R)-\depth(R)$. 
\pushQED{\qed} 
\qedhere
\popQED	
\end{chunk}

\begin{chunk} \textbf{Weak Reducible Complexity.} \label{rcx} (\cite{Be}) Let $M, N \in \md R$. Consider a homogeneous element $\eta$ in the graded module
$\Ext^*_R(M,N)=\bigoplus^{\infty}_{i=0}\Ext^i_R(M,N)$. Then choose a map $f_{\eta}:\Omega^{|\eta|}(M)\rightarrow N$ representing $\eta$, where $\Omega(M)$ denotes the syzygy of $M$ and $|\eta|$ denotes the degree of $\eta$ in $\Ext^*_R(M,N)$. This yields a commutative diagram with exact rows:
\begin{center}
$\begin{CD}
&&&&&&&&\\
  \ \ &&&&  0@>>>\Omega^{|\eta|}(M) @>>> F_{|\eta|-1}@>>>\Omega^{|\eta|-1}(M) @>>>0&  \\
                                &&&&&& @VV{f_{\eta}}V @VV V @VV{\parallel} V\\
  \ \  &&&& 0@>>> N @>>> K_{\eta} @>>>\Omega^{|\eta|-1}(M)@>>>0.&\\
\end{CD}$
\end{center}
\vspace{0.05in}

Here $K_{\eta}$ is the pushout of $f_{\eta}$ and the inclusion $\Omega^{|\eta|}(M)\hookrightarrow F_{|\eta|-1}$. Note the module $K_{\eta}$ is independent, up to isomorphism, of the map $f_{\eta}$ chosen to represent ${\eta}$.

The full subcategory of $\md R$ consisting of modules having
\emph{weak-reducible complexity} is defined inductively as follows:
\begin{enumerate}[\rm(i)]
\item Each module in $\md R$ of finite projective dimension has weak-reducible complexity.
\item If $X\in \md R$ is a module with $0<\cx_R(X)<\infty$, then $X$ has weak-reducible complexity provided that
there exists a homogeneous element $\eta\in\Ext^{*}_R(X,X)$, of positive degree, such that $\cx_R(K_{\eta}) < \cx_R(X)$, and $K_{\eta}$ has weak-reducible complexity.\pushQED{\qed} 
\qedhere
\popQED	
\end{enumerate}
\end{chunk}

\begin{chunk} \textbf{Reducible Complexity.} \label{rcx2} (\cite{Be}) A module $X\in \md R$ has \emph{reducible complexity} if it has weak-reducible complexity and $\depth_R(M)=\depth_R(K_{\eta})$, 
where $K_{\eta}$ is the module discussed in \ref{rcx}. Therefore, over Cohen-Macaulay local rings, the class of modules having weak reducible complexity coincide with the class of modules with reducible complexity. \pushQED{\qed} 
\qedhere
\popQED	
\end{chunk}

\begin{chunk} \textbf{Complete Intersection Dimension Versus Reducible Complexity.} \label{rcx3} If $M\in \md R$ has finite complete intersection dimension, then it has reducible complexity; see \cite[2.2(i)]{Be}. On the other hand, there are modules $M\in \md R$ having reducible complexity with infinite complete intersection dimension satisfying $\Ext^i_R(M,R)=0$ for all $i\gg 0$; see for example \cite[Example on page 136]{Be} and \cite[3.1]{GAV}. \pushQED{\qed} 
\qedhere
\popQED	
\end{chunk}

\section{Main Result}

Bergh \cite[2.2(ii)]{Be} showed that, if $R$ is a Cohen-Macaulay local ring and $M\in \md R$ has reducible complexity, then so does $\Omega^i(M)$ for each $i\geq 0$; his argument in fact implies that the Cohen-Macaulay assumption can be removed for certain values of $i$. More precisely, Bergh's result implies:

\begin{chunk} (\cite[2.2(ii)]{Be}) \label{lem:reduceible} Let $M\in \md R$ be a module that has weak-reducible complexity. 
\begin{enumerate}[\rm(i)]
\item Then $\Omega^{i}(M)$ has weak-reducible complexity for each integer $i\geq0$. 
\item If $t=\depth(R)-\depth_R(M)\geq 2$, then $\Omega^{i}(M)$ has reducible complexity for each $i=1, \ldots, t-1$. \pushQED{\qed} 
\qedhere
\popQED	
\end{enumerate}
\end{chunk}
We will also need another result of Bergh:

\begin{chunk} (\cite[3.1]{Be}) \label{lem:reduceible2} Let $M\in \md R$ be a module that has reducible complexity. If $\Ext^i_R(M,R)=0$ for all $i\gg0$, then it follows that $\depth(R)-\depth_R(M)=\sup\{i\in \ZZ \mid\Ext^i_R(M,R)\neq0\}$.  
\pushQED{\qed} 
\qedhere
\popQED	
\end{chunk}

The proof of Theorem \ref{prp:depth-reducible} relies on the following technical result whose proof is quite involved, and hence deferred to the end of this section.

\begin{chunk} \label{p3} Let $M, N\in \md R$ be nonzero modules. Assume $M$ has weak-reducible complexity. Assume further $\Tor_i^R(M,N)=0=\Ext^i_R(M,R)$ for all $i\geq 1$. Then it follows $\depth_R(M\otimes_RN)=\depth_R(N)$ and that $\Ext^i_R(\Tr M,N)=0$ for all $i\geq 1$.
\pushQED{\qed} 
\qedhere
\popQED	
\end{chunk}

Next is our main result, which is a generalization of Theorem \ref{thmold} advertised in the introduction. Recall that, if $R$ is a local ring and $M\in \md R$ is a module with $\CI_R(M)<\infty$, then $M$ has reducible complexity and $\Ext_R^i(M,R)=0$ for all $i\gg0$, but not vice versa, in general.

\begin{thm}\label{prp:depth-reducible}
Let $M, N \in \md R$ be modules. Assume $\Ext^j_R(M,R)=0$ for all $j\gg0$. Assume further $M$ has reducible complexity. If $\Tor_{i}^R(M,N)=0$ for all $i\geq1$, then the depth formula for $M$ and $N$ holds, i.e., 
  \begin{equation}\notag{}
    \depth_R(M)+\depth_R(N)=\depth(R)+\depth_R(M\otimes_R N).
  \end{equation}
\end{thm}

\begin{proof} We may assume both $M$ and $N$ are nonzero. Set $t=\depth R-\depth_R(M)$, and proceed by induction on $t$. Note that, by \ref{lem:reduceible2}, we have $t=\sup\{i\in \ZZ \mid\Ext^i_R(M,R)\neq0\}$. Moreover, we may assume $t\geq 1$ as if $t=0$, then the assertion follows from \ref{p3}.

Now we argue by induction on $\cx_R(M)$. If $\cx_R(M)=0$, then $\pd_R(M)<\infty$, and so the depth formula holds by Theorem \ref{thmold}. Hence we assume assume $\pd_R(M)=\infty$, i.e., $\cx_R(M)\geq 1$. As $M$ has reducible complexity, there exists a short exact sequence 
\begin{equation}\tag{\ref{prp:depth-reducible}.1}
0 \to M \to K \to \Omega^n(M) \to 0,
\end{equation}
where $n$ is a nonnegative integer, $K\in \md R$ has reducible complexity, $\cx_R(K)<\cx_R(M)$ and $\depth_R(K)=\depth_R(M)$. Note, it follows from (\ref{prp:depth-reducible}.1) that $\Ext^j_R(K,R)=0$ for all $j\gg0$, and $\Tor_{i}^R(K,N)=0$ for all $i\geq1$. So the induction hypothesis on the complexity gives the equality:
\begin{equation}\tag{\ref{prp:depth-reducible}.2}
  \depth_R(K)+\depth_R(N)=\depth(R)+\depth_R(K\otimes_R N).
\end{equation}
Note, since $\Tor_{i}^R(M,N)=0$ for all $i\geq1$, tensoring (\ref{prp:depth-reducible}.1) with $N$, we obtain the exact sequence:
\begin{equation}\tag{\ref{prp:depth-reducible}.3}
0 \to M\otimes_R N \to K\otimes_R N \to \Omega^n(M) \otimes_R N \to 0.
\end{equation}

Next we will consider cases for the nonnegative integer $n$:

\noindent \textbf{Case 1.} Assume $n=0$. Then $\Omega^n(M)=M$, and the depth lemma applied to the short exact sequence (\ref{prp:depth-reducible}.3) yields $\depth_R(M\otimes_R N)=\depth_R(K\otimes_R N)$. So, the depth formula for $M$ and $N$ holds by (\ref{prp:depth-reducible}.2).

For the remaining cases, we will make use of the following observation; it follows easily from the depth lemma and (\ref{prp:depth-reducible}.3).
\begin{equation}\tag{\ref{prp:depth-reducible}.4}
\text{The proof of the theorem is complete in case } \depth_R(\Omega^n(M)\otimes_R N)>\depth_R(K\otimes_R N).
\end{equation}


\noindent \textbf{Case 2.} Assume $1\leq n \leq t-1$. In this case, by \ref{lem:reduceible}, we know $\Omega^n(M)$ has reducible complexity.
Since $\depth(R)-\depth_R(M)=t\geq 2$, we have $\depth_R(\Omega^n(M))=\depth_R(M)+v$ for some positive integer $v$ with $1\leq v \leq t-1$. Hence, $\depth(R)-\depth_R(\Omega^n(M))=t-v<t$.
Now, by replacing the pair $(M,N)$ with $(\Omega^n(M),N)$, and by using the induction hypothesis on $t$, we obtain:
\begin{equation}\tag{\ref{prp:depth-reducible}.5}
  \depth_R(\Omega^n(M))+\depth_R(N)=\depth (R)+\depth_R(\Omega^n(M) \otimes_R N).
\end{equation}
Thus, since $\depth_R(\Omega^n(M))=\depth_R(M)+v$, we conclude from (\ref{prp:depth-reducible}.2) and (\ref{prp:depth-reducible}.5) that:
\begin{equation}\tag{\ref{prp:depth-reducible}.6}
\depth_R(\Omega^n(M) \otimes_R N)=v+\depth(N)-t=v+\depth(K\otimes_RN).
\end{equation}
In particular, we see from (\ref{prp:depth-reducible}.6) that:
\begin{equation}\tag{\ref{prp:depth-reducible}.7}
\depth_R(\Omega^n(M)\otimes_R N)>\depth_R(K\otimes_R N).
\end{equation}
Hence the required result follows due to (\ref{prp:depth-reducible}.4).

\noindent \textbf{Case 3.} Assume $n\geq t$. Notice, by \ref{lem:reduceible}, $\Omega^{n}(M)$ has weak-reducible complexity.
Hence \ref{p3} implies that:
\begin{equation}\tag{\ref{prp:depth-reducible}.8}
\depth_R(\Omega^{n}(M) \otimes_RN)=\depth_R(N).
\end{equation}
Therefore, since $t\geq 1$, (\ref{prp:depth-reducible}.2) and (\ref{prp:depth-reducible}.8) yield that: 
\begin{equation}\notag{} 
\depth_R(K\otimes_RN)<\depth_R(K\otimes_RN)+t=\depth_R(N)=\depth_R(\Omega^{n}(M) \otimes_RN).
\end{equation}
Thus the proof of Case 3, as well as the proof of the theorem, is complete by (\ref{prp:depth-reducible}.4).
\end{proof}

We now proceed to establish \ref{p3} and complete the proof of Theorem \ref{prp:depth-reducible}. For that we will make use of the following results, which are recorded here for the convenience of the reader.

\begin{chunk}\label{ArY} (\cite[4.1]{ArY}) Let $X, Y \in \md R$ be modules such that $\Ext^i_R(X,Y)=0$ for all $i\geq 1$. Then it follows that $\depth_R(\Hom_R(X,Y))=\depth_R(Y)$. \pushQED{\qed} 
\qedhere
\popQED	
\end{chunk}

\begin{chunk}\label{dualize} (\cite[3.9]{AuBr}) Let $0\to A \to B \to C \to 0$ be a short exact sequence in $\md(R)$. Then it follows that the sequence
$0\to C^{\ast} \to B^{\ast} \to C^{\ast} \to \Tr A \to \Tr B \to \Tr C \to 0$ is exact.  \pushQED{\qed} 
\qedhere
\popQED	

\end{chunk}

\begin{chunk}\label{transpose} Let $M, N \in \md R$ be modules and let $n\geq 1$ be an integer. Assume $\Ext_R^i(M,R)=0$ for all $i\geq n$. Then, for each integer $j$ with $ j\geq n$, we have 
$\Ext^j_R(M,N)\cong\Tor_1^R(\Tr\Omega^jM,N)$ and $\Tr \Omega ^{j-1} M \cong \Omega \Tr \Omega^j M$ (up to free summands); see \cite[2.8]{AuBr} for details. \pushQED{\qed} 
\qedhere
\popQED	
\end{chunk}

\begin{chunk}\label{Bee} (\cite[2.3 and 2.4(i)]{Be}; see also \cite[2.1(ii)]{Be1}) Let $M\in \md R$ and let $\eta\in\Ext^{|\eta|}_R(M,M)$ be an element. 
\begin{enumerate}[\rm(i)]
\item There is an exact sequence $0\to \Omega^{|\eta|} (K_{\eta}) \to K_{\eta^2} \oplus F \to K_{\eta} \to 0$ in $\md R$, where $F$ is a free module.
\item Assume $K_{\eta}$ reduces the complexity of $M$. Then it follows that: $$\cx_R(K_{\eta^2})=\cx_R(K_{\eta^2}\oplus F)\leq \max\{\cx_R(\Omega^{|\eta|} (K_{\eta})), \cx_R(K_{\eta})\}=\cx_R(K_{\eta})<\cx_R(M).$$
Therefore, there is an exact sequence of the form $0 \to M \to K_{\eta^2} \to \Omega^{2|\eta|-1}(M)\to 0$, where $K_{\eta^2}$ also reduces the complexity of $M$. \pushQED{\qed} 
\qedhere
\popQED	
\end{enumerate}

\end{chunk}

\begin{rmk} In \cite[2.4(i)]{Be} it is assumed that the ring in question is a complete intersection. Also, in \cite[2.1(ii)]{Be1}, it is assumed that the module $M$ considered has finite complete intersection dimension. Although we refer to \cite[2.4(i)]{Be} (or \cite[2.1(ii)]{Be1}) in the proof of \ref{p3}, we do not need that rings are complete intersections or modules have finite complete intersection dimension  in the context of our argument; see \ref{Bee}. \pushQED{\qed} 
\qedhere
\popQED	
\end{rmk}

\begin{proof}[A Proof of \ref{p3}]
We set $c=\cx_R(M)$, and proceed by induction on $c$. 

Assume $c=0$, i.e., $\pd_R(M)<\infty$. Then, since $\Ext^i_R(M,R)=0$ for all $i\geq 1$, it follows that $M$ is free. Therefore, $\Tr M=0$ and the claim follows.

Next assume $c\geq 1$. As $\Tor_i^R(M,N)=0$ for all $i\geq 1$, it follows from (\ref{a1}.1) that 
\begin{equation}\tag{\ref{p3}.1}
\Ext^1_R(\Tr\Omega^{i}M,N)=0 \text{ for all } i\geq 1. 
\end{equation}
Moreover, since $\Ext^i_R(M,R)=0$ for all $i\geq 1$, the following stable isomorphism is deduced from \ref{transpose}:
\begin{equation}\tag{\ref{p3}.2}
\Tr\Omega^{u-v}M\cong\Omega^v\Tr\Omega^{u}M  \text{, for all positive integers $u$ and $v$ with } u \geq v.
\end{equation}
Therefore, for a given integer $t\geq 2$ and $1\leq j \leq t-1$, we have that:
\begin{equation}\tag{\ref{p3}.3}
\Ext^j_R(\Tr\Omega^{t-1}M,N)\cong \Ext^1_R(\Omega^{j-1}\Tr\Omega^{t-1}M, N) \cong \Ext^1_R(\Tr\Omega^{t-j}M, N) =0.
\end{equation}
The second isomorphism and the first equality in (\ref{p3}.3) are due to (\ref{p3}.2) and (\ref{p3}.1), respectively.

Now let $\eta\in\Ext^*_R(M,M)$ be an element reducing the complexity of $M$; see \ref{rcx}. Hence, there is an exact sequence of the form:
\begin{equation}\tag{\ref{p3}.4}
0\rightarrow M\rightarrow K \rightarrow\Omega^q(M)\rightarrow0,
\end{equation}
where $q=|\eta|-1$, $K=K_{\eta}$, $\cx_R(K)<c$ and $K$ has weak-reducible complexity. As $\Tor_i^R(M,N)=0=\Ext^i_R(M,R)$ for all $i\geq 1$, it follows from (\ref{p3}.4) that $\Ext^i_R(K,R)=0=\Tor_i^R(K,N)$ for all $i\geq 1$. So, by the induction hypothesis, we conclude:
\begin{equation}\tag{\ref{p3}.5}
\Ext^i_R(\Tr K,N)=0 \text{ for all } i\geq 1 \text{, and } \depth_R(K\otimes_RN)=\depth_R(N).
\end{equation}

We proceed to prove the required assertions, i.e., the vanishing of $\Ext^i_R(\Tr M,N)$ for all $i\geq 1$ and the depth equality $\depth_R(M\otimes_RN)=\depth_R(N)$, in several steps. 
\vspace{0.1in}


\noindent \textbf{Claim 1.} We have that $\Ext^i_R(\Tr\Omega^{q}(M),N)\cong\Ext^{i+1}_R(\Tr M,N)  \cong     \Ext^{i+q+1}_R(\Tr\Omega^{q}(M),N)$ for all $i\geq 1$.\\
\noindent \textbf{Proof of Claim 1.} The short exact sequence (\ref{p3}.4), in view of \ref{dualize}, yields the exact sequence:
\begin{equation}\tag{\ref{p3}.6}
0\rightarrow {(\Omega^q(M))}^*\rightarrow {K}^*\rightarrow M^*\rightarrow \Tr\Omega^q(M)\rightarrow \Tr K_{\eta}\rightarrow \Tr M\rightarrow0.
\end{equation}
Since $\Ext^1_R(\Omega^q(M),R)=0$, the following sequence is exact:
\begin{equation}\tag{\ref{p3}.7}
0\rightarrow \Tr\Omega^q(M)\rightarrow \Tr K\rightarrow \Tr M\rightarrow0.
\end{equation}
We obtain, by applying $\Hom_R(-,N)$ to (\ref{p3}.7), the following long exact sequence:
\begin{equation}\tag{\ref{p3}.8}
\cdots\rightarrow\Ext^i_R(\Tr M,N)\rightarrow\Ext^i_R(\Tr K,N)\rightarrow\Ext^i_R(\Tr \Omega^q(M),N)\rightarrow\cdots
\end{equation}
Now (\ref{p3}.8) and (\ref{p3}.5) give:
\begin{equation}\tag{\ref{p3}.9}
\Ext^{i+1}_R(\Tr M,N)\cong\Ext^i_R(\Tr\Omega^{q}(M),N) \text{ for all } i\geq 1. 
\end{equation}
Consequently, for all $i\geq 1$, we establish:
\begin{equation}\tag{\ref{p3}.10}
\Ext^i_R(\Tr\Omega^{q}(M),N)\cong\Ext^{i+1}_R(\Tr M,N)  \cong   \Ext^{i+1}_R(\Omega^q \Tr \Omega^q(M),N)     \cong    \Ext^{i+q+1}_R(\Tr\Omega^{q}(M),N).  
\end{equation}
Here, in (\ref{p3}.10), the first and second isomorphisms are due to (\ref{p3}.9) and (\ref{p3}.2), respectively.
This completes the proof of Claim 1.  \pushQED{\qed} 
\qedhere
\popQED	
\vspace{0.1in}


\noindent \textbf{Claim 2.} We have that $\Ext^{i}_R(\Tr\Omega^{q}(M),N) \cong \Ext^{i+j(q+1)}_R(\Tr\Omega^{q}(M),N)$ for all $i\geq 1$ and $j\geq 1$.\\
\noindent \textbf{Proof of Claim 2.}  This follows by repeated applications of Claim 1. \pushQED{\qed} 
\qedhere
\popQED	
\vspace{0.1in}

\noindent \textbf{Claim 3.} We have that $\Ext^i_R(\Tr\Omega^{2q+1}(M),N)\cong\Ext^{i+1}_R(\Tr M,N)\cong\Ext^{2q+i+2}_R(\Tr\Omega^{2q+1}M,N)$ for all $i\geq 1$.\\
\noindent \textbf{Proof of Claim 3.} It follows that $\eta^2$ reduces the complexity of $M$, and there are exact sequences:
\begin{equation}\tag{\ref{p3}.11}
0 \rightarrow M \rightarrow Z \rightarrow \Omega^{2q+1}(M) \rightarrow 0,
\end{equation}
and
\begin{equation}\tag{\ref{p3}.12}
0\rightarrow\Omega^{q+1}(K)\rightarrow Z\oplus F \rightarrow K \rightarrow 0,
\end{equation}
where $Z=K_{\eta^2}$ and $F$ is a free module; see \ref{rcx} and \ref{Bee}.

As $\Ext^1_R(K,R)=0$, the following exact sequence follows from (\ref{p3}.12) and \ref{dualize}:
\begin{equation}\tag{\ref{p3}.13}
0\rightarrow\Tr K \rightarrow\Tr Z \rightarrow\Tr \Omega^{q+1}(K) \rightarrow 0.
\end{equation}
Applying $\Hom_R(-,N)$ to the exact sequence (\ref{p3}.13), we get a long exact sequence:
\begin{equation}\tag{\ref{p3}.14}
\cdots \rightarrow \Ext^i_R(\Tr \Omega^{q+1}(K),N) \rightarrow \Ext^i_R(\Tr Z, N) \rightarrow \Ext^i_R(\Tr K,N)\rightarrow\cdots
\end{equation}

Note that $\Omega^{q+1}(K)$ has weak-reducible complexity; see \ref{lem:reduceible}(i). Note also $\cx_R(\Omega^{q+1}(K))=\cx_R(K)<c$, and $\Ext^i_R(\Omega^{q+1}(K),R)=0=\Tor_i^R(\Omega^{q+1}(K),N)$ for all $i\geq 1$. Therefore, by the induction hypothesis on $c$, we have that $\Ext^i_R(\Tr\Omega^{q+1}(K),N)=0$ for all $i\geq 1$. In view of (\ref{p3}.5) and (\ref{p3}.14), we conclude: 
\begin{equation}\tag{\ref{p3}.15}
\Ext^i_R(\Tr Z, N)=0 \text{ for all } i\geq 1.
\end{equation}
The short exact sequence (\ref{p3}.11) and \ref{dualize} yield the following exact sequence: 
\begin{equation}\tag{\ref{p3}.16}
0\rightarrow(\Omega^{2q+1}M)^*\rightarrow (Z)^*\rightarrow M^*\rightarrow\Tr\Omega^{2q+1}(M) \rightarrow\Tr Z \rightarrow\Tr M\rightarrow0.
\end{equation}
Since we have $\Ext^{2q+2}_R(M,R)=0$, by (\ref{p3}.16), we get the exact sequence:
\begin{equation}\tag{\ref{p3}.17}
0\rightarrow\Tr\Omega^{2q+1}(M) \rightarrow\Tr Z \rightarrow\Tr M\rightarrow0.
\end{equation}
Now (\ref{p3}.17) induces the long exact sequence for all $i\geq 1$:
\begin{equation}\tag{\ref{p3}.18}
\cdots\rightarrow\Ext^i_R(\Tr M,N)\rightarrow\Ext^i_R(\Tr Z,N)\rightarrow\Ext^i_R(\Tr\Omega^{2q+1}(M),N)\rightarrow\cdots
\end{equation}
Consequently, for all $i\geq 1$, we have:
\begin{align}\tag{\ref{p3}.19}
\Ext^i_R(\Tr\Omega^{2q+1}(M),N) & \cong \Ext^{i+1}_R(\Tr M,N)  \\ & \notag{} \cong \Ext^{i+1}_R(\Omega^{2q+1}\Tr \Omega^{2q+1}(M),N) \\ & \notag{} \cong \Ext^{2q+i+2}_R(\Tr\Omega^{2q+1}M,N) 
\end{align}
Here, in (\ref{p3}.19), the first isomorphism follows from the long exact sequence in (\ref{p3}.18) since $\Ext^i_R(\Tr Z, N)$ vanishes for all $i\geq 1$; see (\ref{p3}.15). Furthermore, the second isomorphism of 
(\ref{p3}.19) is due to (\ref{p3}.2). This completes the proof of Claim 3.  \pushQED{\qed} 
\qedhere
\popQED	
\vspace{0.1in}


\noindent \textbf{Claim 4.} Assume $q\geq 1$. Then, given $j\geq 1$, we have that $\Ext^{i}_R(\Tr\Omega^{q}(M),N)=0$ for all $i\neq j(q+1)$, i.e., $\Ext^{i}_R(\Tr\Omega^{q}(M),N)=0$ for all $i$, where $(j-1)q+j \leq i \leq jq+(j-1)$.\\
\noindent \textbf{Proof of Claim 4.} Let $j\geq 1$ be an integer.

If $j=1$, then setting $t=q+1$ in (\ref{p3}.3), we see that $\Ext^{i}_R(\Tr\Omega^{q}(M),N)=0$ for all $i$ with $1\leq i \leq q$. Hence assume $j\geq 2$. In this case, we have $1 \leq i-(j-1)(q+1)\leq q$, and Claim 2 implies that:
\begin{equation}\tag{\ref{p3}.20}
\Ext^{i}_R(\Tr\Omega^{q}(M),N) \cong \Ext^{i-(j-1)(q+1)}_R(\Tr\Omega^{q}(M),N)
\end{equation}
We have observed $\Ext^{v}_R(\Tr\Omega^{q}(M),N)=0$ for all $v$ with $1 \leq v \leq q$. Thus, since $1 \leq i-(j-1)(q+1)\leq q$, we see that $\Ext^{i-(j-1)(q+1)}_R(\Tr\Omega^{q}(M),N)=0$.
Therefore Claim 4 follows from (\ref{p3}.20).
\pushQED{\qed} 
\qedhere
\popQED	
\vspace{0.1in}

\noindent \textbf{Claim 5.} If $q\geq 1$, then we have that $\Ext^i_R(\Tr \Omega ^qM,N)=0$ for all $i\geq 1$.\\
\noindent \textbf{Proof of Claim 5.} We have:
\begin{align}\tag{\ref{p3}.21}
0=\Ext^{q+1}_R(\Tr\Omega^{2q+1}(M),N) & \cong \Ext^{q+2}_R(\Tr M,N)  \cong \Ext^{q+1}_R(\Tr \Omega^q(M),N) & \notag{} 
\end{align}
Here, in (\ref{p3}.21), the first equality follows from (\ref{p3}.3) by letting $t=2q+2$ and $j=q+1$. Furthermore, the first and second isomorphisms are due to Claim 1 and Claim 3 (with $i=q+1$), respectively. 

Claim 2, in view of (\ref{p3}.21), implies that $0=\Ext^{q+1}_R(\Tr \Omega^q(M),N)\cong \Ext^{(q+1)+j(q+1)}_R(\Tr \Omega^q(M),N)$ for all $j\geq 1$, i.e., $\Ext^{r(q+1)}_R(\Tr \Omega^q(M),N)=0$ for all $r\geq 1$. This observation, in combination with Claim 4, establishes Claim 5.
\pushQED{\qed} 
\qedhere
\popQED	
\vspace{0.1in}


\noindent \textbf{Claim 6.} We have that $\Ext^i_R(\Tr M,N)=0$ for all $i\geq 1$.\\
\noindent \textbf{Proof of Claim 6.} Assume first $q=0$. Then, for all $i\geq 1$, we have:
\begin{align}\tag{\ref{p3}.22}
\Ext^{i+1}_R(\Tr M,N) \cong \Ext^{i}_R(\Tr M,N) \cong \Ext^{i}_R(\Tr \Omega M,N) 
\end{align}
Here, in (\ref{p3}.22), the first and the second isomorphism follows from Claim 1 and Claim 3, respectively. Since $\Ext^1_R(\Tr \Omega M,N)$ vanishes due to (\ref{p3}.1),
we conclude that $\Ext^i_R(\Tr M,N)=0$ for all $\geq 1$.

Next assume $q\geq 1$. Then, for all $i\geq 1$, we have: 
\begin{align}\tag{\ref{p3}.22}
0=\Ext^i_R(\Tr\Omega^{q}(M),N)\cong\Ext^{i+1}_R(\Tr M,N) 
\end{align}
In (\ref{p3}.22), the first equality is due to Claim 5, while the first isomorphism follows from Claim 1. Consequently, we have $\Ext^{i}_R(\Tr M,N)=0$ for all $i\geq 2$. Furthermore, it follows:
\begin{align}\tag{\ref{p3}.23}
\Ext^1_R(\Tr M,N)\cong \Ext^{1}_R(\Omega^q \Tr \Omega^q M,N) \cong  \Ext^{q+1}_R(\Tr \Omega^q M,N)=0 
\end{align}
The first isomorphism of (\ref{p3}.23) is due to (\ref{p3}.2), and the first equality is from Claim 5. This proves the vanishing of $\Ext^i_R(\Tr M,N)$ for all $i\geq 1$, and completes the proof of Claim 6.
\pushQED{\qed} 
\qedhere
\popQED	
\vspace{0.1in}


\noindent \textbf{Claim 7.} We have that $\depth_R(M\otimes_RN)=\depth_R(N)$.\\
\noindent \textbf{Proof of Claim 7.} Recall that $M^{\ast} \cong \Omega^2 \Tr M \oplus G$ for some free module $G \in \md R$; see \ref{a1}. Therefore, as Claim 6 shows $\Ext^i_R(\Tr M,N)=0$ for all $\geq 1$, we conclude that $\Ext^i_R(M^{\ast},N)=0$ for all $\geq 1$. This implies, in view of \ref{ArY}, that:
\begin{align}\tag{\ref{p3}.24}
\depth_R(\Hom(M^{\ast},N))=\depth_R(N)
\end{align}
On the other hand, since $\Ext^1_R(\Tr M,N)=0=\Ext^2_R(\Tr M,N)=0$, setting $n=0$, we obtain from \ref{a1}.1 that:
\begin{align}\tag{\ref{p3}.25}
M\otimes_R N \cong \Hom_R(M^{\ast},N)
\end{align}
Consequently, the proof of Claim 7 is complete due to (\ref{p3}.24) and (\ref{p3}.25).
\end{proof}


\def\cfudot#1{\ifmmode\setbox7\hbox{$\accent"5E#1$}\else
  \setbox7\hbox{\accent"5E#1}\penalty 10000\relax\fi\raise 1\ht7
  \hbox{\raise.1ex\hbox to 1\wd7{\hss.\hss}}\penalty 10000 \hskip-1\wd7\penalty
  10000\box7}
\begin{thebibliography}{10}

\bibitem{Alp}
Jonathan~Lazare Alperin.
\newblock Periodicity in groups.
\newblock {\em Illinois J. Math.}, 21(4):776--783, 1977.

\bibitem{AE}
Jonathan~Lazare Alperin and Leonard Evens.
\newblock Representations, resolutions and quillen's dimension theorem.
\newblock {\em Journal of Pure and Applied Algebra}, 22(1):1--9, Aug 1981.

\bibitem{ArY}
Tokuji Araya and Yuji Yoshino.
\newblock Remarks on a depth formula, a grade inequality and a conjecture of
  {A}uslander.
\newblock {\em Comm. Algebra}, 26(11):3793--3806, 1998.

\bibitem{Au}
Maurice Auslander.
\newblock Modules over unramified regular local rings.
\newblock {\em Illinois J. Math.}, 5:631--647, 1961.

\bibitem{AuBr}
Maurice Auslander and Mark Bridger.
\newblock {\em Stable module theory}.
\newblock Memoirs of the American Mathematical Society, No. 94. American
  Mathematical Society, Providence, R.I., 1969.

\bibitem{Av1}
Luchezar Avramov.
\newblock Modules of finite virtual projective dimension.
\newblock {\em Invent. Math.}, 96(1):71--101, 1989.

\bibitem{AvBu}
Luchezar Avramov and Ragnar-O. Buchweitz.
\newblock Support varieties and cohomology over complete intersections.
\newblock {\em Invent. Math.}, 142(2):285--318, 2000.

\bibitem{AGP}
Luchezar Avramov, Vesselin Gasharov, and Irena Peeva.
\newblock Complete intersection dimension.
\newblock {\em Inst. Hautes \'Etudes Sci. Publ. Math.}, (86):67--114 (1998),
  1997.

\bibitem{Be}
Petter~Andreas Bergh.
\newblock Modules with reducible complexity.
\newblock {\em Journal of Algebra}, 310(1):132--147, 2007.

\bibitem{Be1}
Petter~Andreas Bergh.
\newblock On the vanishing of (co)homology over local rings.
\newblock {\em J. Pure Appl. Algebra}, 212(1):262--270, 2008.

\bibitem{CJ}
Lars~Winther Christensen and David Jorgensen.
\newblock Vanishing of {T}ate homology and depth formulas over local rings.
\newblock {\em Journal of Pure and Applied Algebra}, 219(3):464--481, 2015.

\bibitem{Foxby}
Hans-Bj{\o}rn Foxby.
\newblock Homological dimensions of complexes of modules.
\newblock In {\em S\'eminaire d'{A}lg\`ebre {P}aul {D}ubreil et {M}arie-{P}aule
  {M}alliavin, 32\`eme ann\'ee ({P}aris, 1979)}, volume 795 of {\em Lecture
  Notes in Math.}, pages 360--368. Springer, Berlin, 1980.

\bibitem{GAV}
Vesselin Gasharov and Irena Peeva.
\newblock Boundedness versus periodicity over commutative local rings.
\newblock {\em Transactions of the American Mathematical Society},
  320(2):569--580, 1990.

\bibitem{HW1}
Craig Huneke and Roger Wiegand.
\newblock Tensor products of modules and the rigidity of {T}or.
\newblock {\em Math. Ann.}, 299(3):449--476, 1994.

\bibitem{I}
Srikanth Iyengar.
\newblock Depth for complexes, and intersection theorems.
\newblock {\em Math. Z.}, 230(3):545--567, 1999.

\end{thebibliography}

\end{document}